\documentclass[a4paper,12pt]{article}
\usepackage{amssymb,amsmath,amsthm,latexsym}
\usepackage{amsfonts}
	\usepackage{amsfonts}
\usepackage{graphicx}
\usepackage[pdftex,bookmarks,colorlinks=false]{hyperref}
\usepackage{verbatim}
\usepackage{caption}
\usepackage{subcaption}
\usepackage[hmargin=1.15in,vmargin=1.15in]{geometry}

\newtheorem{theorem}{Theorem}[section]

\newtheorem{definition}[theorem]{Definition}

\newtheorem{lemma} [theorem]{Lemma}

\newtheorem{proposition}[theorem]{Proposition}
\newtheorem{remark}[theorem]{Remark}

\title{\bf {\sc The Sparing Number of the Cartesian Products of Certain Graphs}}

\author{{\bf K P Chithra\footnote{Naduvath Mana, Nandikkara, Thrissur-680301, India. {\em chithrasudev@gmail.com} }},~ {\bf K A Germina\footnote{Department of Mathematics, School of Mathematical \& Physical Sciences, Central University of Kerala, Kasaragod-671316, email:{\em srgerminaka@gmail.com}}} ~and~ {\bf N K Sudev\footnote{Department of Mathematics, Vidya Academy of Science \& Technology, Thalakkottukara, Thrissur - 680501, India. email: {\em sudevnk@gmail.com}}}}
\date{}
\begin{document}
\maketitle

\begin{abstract}
Let $\mathbb{N}_0$ be the set of all non-negative integers. An integer additive set-indexer (IASI) is defined as an injective function $f:V(G)\rightarrow \mathcal{P}(\mathbb{N}_0)$ such that the induced function $f^+:E(G) \rightarrow \mathcal{P}(\mathbb{N}_0)$ defined by $f^+ (uv) = f(u)+ f(v)$ is also injective, where $f(u)+f(v)$ is the sumset of $f(u)$ and $f(v)$ and $\mathcal{P}(\mathbb{N}_0)$ is the power set of $\mathbb{N}_0$. If $f^+(uv)=k~\forall~uv\in E(G)$, then $f$ is said to be a $k$-uniform integer additive set-indexer. An integer additive set-indexer $f$ is said to be a weak integer additive set-indexer if $|f^+(uv)|=max(|f(u)|,|f(v)|)~\forall ~ uv\in E(G)$. In this paper, we study about the sparing number of the cartesian product of two graphs. 
\end{abstract}

\noindent \textbf{Key Words:} Integer additive set-indexers, mono-indexed elements of a graphs, weak integer additive set-indexers, sparing number of a graph.
\newline
\textbf{AMS Subject Classification: 05C78}

\section{Introduction}
For all  terms and definitions, not defined specifically in this paper, we refer to \cite{FH} and for more about graph products we refer to \cite{HIS}. Unless mentioned otherwise, all graphs considered here are simple, finite and have no isolated vertices.

Let $\mathbb{N}_0$ denote the set of all non-negative integers. For all $A, B \subseteq \mathbb{N}_0$, the sum of these sets is denoted by  $A+B$ and is defined by $A + B = \{a+b: a \in A, b \in B\}$. The set $A+B$ is called the {\em sumset} of the sets $A$ and $B$. If either $A$ or $B$ is countably infinite, then their sumset is also countably infinite. Hence, the sets we consider here are all finite sets of non-negative integers. The cardinality of a set $A$ is denoted by $|A|$. The power set of a set $A$ is denoted by $\mathcal{P}(A)$.

\begin{definition}\label{D2}{\rm
\cite{GA} An {\em integer additive set-indexer} (IASI, in short) is defined as an injective function $f:V(G)\rightarrow \mathcal{P}(\mathbb{N}_0)$ such that the induced function $f^+:E(G) \rightarrow \mathcal{P}(\mathbb{N}_0)$ defined by $f^+ (uv) = f(u)+ f(v)$ is also injective}.
\end{definition}

\begin{lemma}
\cite{GS1} If $f$ is an integer additive set-indexer on a graph $G$, then $max(|f(u)|,\\|f(v)|) \le f^+(uv) \le |f(u)||f(v)|, \forall ~ u,v\in  V(G)$.
\end{lemma}

\begin{definition}{\rm
\cite{GS1} An IASI $f$ is said to be a {\em weak IASI} if $|f^+(uv)|=max(|f(u)|,\\|f(v)|)$ for all $u,v\in V(G)$. A weak IASI $f$ is said to be {\em weakly uniform IASI} if $|f^+(uv)|=k$, for all $u,v\in V(G)$ and for some positive integer $k$.  A graph which admits a weak IASI may be called a {\em weak IASI graph}.}
\end{definition}

It is to be noted that if $G$ is a weak IASI graph, then every edge of $G$ has at least one mono-indexed end vertex(or, equivalently no two adjacent vertices can have non-singleton set-labels simultaneously).

\begin{definition}{\rm
\cite{GS3} The cardinality of the set-label of an element (vertex or edge) of a graph $G$ is called the {\em set-indexing number} of that element. An element (a vertex or an edge) of a graph which has the set-indexing number $1$ is called a {\em mono-indexed element} of that graph.}
\end{definition}

\begin{definition}{\rm
\cite{GS3} The {\em sparing number} of a graph $G$ is defined to be the minimum number of mono-indexed edges required for $G$ to admit a weak IASI and is denoted by $\varphi(G)$.}
\end{definition}

\begin{theorem}\label{T-WSG}
\cite{GS3} A subgraph of a weak IASI graph is also a weak IASI graph.
\end{theorem}

\begin{theorem}\label{T-WUC}
\cite{GS3} A graph $G$ admits a weak IASI if and only if $G$ is bipartite or it has at least one mono-indexed edge. 
\end{theorem}

\begin{theorem}\label{T-WUOC}
\cite{GS3} An odd cycle $C_n$ has a weak IASI if and only if it has at least one mono-indexed edge. 
\end{theorem}

\begin{theorem}\label{T-NME}
\cite{GS3} Let $C_n$ be a cycle of length $n$ which admits a weak IASI, for a positive integer $n$. Then, $C_n$ has an odd number of mono-indexed edges when it is an odd cycle and has even number of mono-indexed edges, when it is an even cycle. 
\end{theorem}

\begin{theorem}\label{T-WKN}
\cite{GS3} The sparing number of complete graph $K_n$ is $\frac{1}{2}(n-1)(n-2)$.
\end{theorem}

In this paper, we discuss about the sparing number of the cartesian products of two weak IASI graphs.

\section{Main Results}

Let $G_1(V_1,E_1)$ and $G_2(V_2,E_2)$ be two graphs.Then, the {\em cartesian product} of $G_1$ and $G_2$, denoted by $G_1\times G_2$, is the graph with vertex set $V_1\times V_2$  defined as follows. Let $u=(u_1, u_2)$ and $v=(v_1,v_2)$ be two points in $V_1\times V_2$. Then, $u$ and $v$ are adjacent in $G_1\times G_2$ whenever [$u_1=v_1$ and $u_2$ is adjacent to $v_2$] or [$u_2=v_2$ and $u_1$ is adjacent to $v_1$]. If $|V_i|=p_i$ and $|E_i|=q_i$ for $i=1,2$, then $|V(G_1\times G_2)|=p_1p_2$ and $|E(G_1\times G_2)|=p_1q_2+p_2q_1$.

\begin{remark}\label{R-CP2G}{\rm
The cartesian product $G_1\times G_2$ may be viewed as follows. Make $p_2$ copies of $G_1$. Denote these copies by $G_{1_i}$, which corresponds to the vertex $v_i$ of $G_2$. Now, join the corresponding vertices of two copies $G_{1_i}$ and $G_{1_j}$ if the corresponding vertices $v_i$ and $v_j$ are adjacent in $G_2$. Thus, we view the product $G_1\times G_2$ as a union of $p_2$ copies of $G_1$ and a finite number of edges connecting two copies $G_{1_i}$ and $G_{1_j}$ of $G_1$ according to the adjacency of the corresponding vertices $v_i$ and $v_j$ in $G_2$, where $1\le i\neq j\le p_2$.}
\end{remark}

\begin{remark}\label{R-PBP}{\rm
The cartesian product of two bipartite graphs is also a bipartite graph. Also, the cartesian products $G_1\times G_2$ and $G_2\times G_1$ of two graphs $G_1$ and $G_2$, are isomorphic graphs.}
\end{remark}

\begin{theorem}
\cite{GS7} Let $G_1$ and $G_2$ be two weak IASI graphs. Then, the product $G_1\times G_2$ also admits a weak IASI. 
\end{theorem}

\begin{theorem}\label{P-WPr-PP1}
The sparing number of a planar grid $P_m\times P_n$ is $0$.
\end{theorem}
\begin{proof}
Let $P_m$ and $P_n$ be two paths which admit weak IASIs . Label the vertices of $P_{m_i}, 1 \le i \le n$, as follows. For odd values of $i$, label the vertices of $P_{m_i}$, starting from the initial vertex, alternately by distinct singleton sets and distinct non-singleton sets respectively and for even values of $i$, label the vertices of $P_{m_i}$, starting from the initial vertex, alternately by non-singleton sets and singleton sets that are not used for labeling any vertex before.  This labeling is a weak IASI for $P_m\times P_n$.

In $P_m\times P_n$, the corresponding vertices of different copies of $P_m$ are adjacent. Hence, if we label as mentioned above, no two edge of $P_m\times P_n$ have the set-label of the same kind. Therefore, the sparing number of a planar grid is $0$.
\end{proof}

Now, the following theorem estimates the sparing number of a {\em prism}, the cartesian product of a cycle and a path.

\begin{proposition}
The sparing number of a prism $C_m\times P_n$ is 
\begin{equation*}
\varphi(C_m\times P_n)=
\begin{cases}
0 & ~~ \text{if $m$ is even}\\
2n+1 & ~~ \text{if $m$ is odd}.
\end{cases}
\end{equation*}
\end{proposition}

\begin{proof}
Since $P_n$ has $n+1$ vertices, there are $n+1$ copies of $C_m$ in $C_m\times P_n$. Now, we consider the following cases.

\noindent {\em Case 1:} Suppose that $m$ is even. Label the vertices of each copy $C_{m_i}$ of $C_m$, starting from the initial vertex, by distinct singleton sets and distinct non-singleton sets alternately for odd number $i$ and label the vertices of $C_{m_i}$, starting from the initial vertex, by distinct non-singleton sets and distinct singleton sets alternately for even number $i$. Then, for every pair of adjacent vertices in $C_m\times P_n$, one will be mono-indexed and the other have non-singleton set-label. Therefore, $\varphi(C_m\times P_n) = 0$.

\noindent {\em Case 2:} Let $m$ be an odd integer. Then, by Theorem \ref{T-WUOC}, $C_m$ has at least one mono-indexed edge. That is, at least two adjacent vertices in each copy of $C_m$ will be mono-indexed.  Then, every copy $C_{m_i}$ of $C_m$ must contain at least one mono-indexed edge. Therefore, if we label the vertices of each $C_{m_i}$ alternately by distinct singleton sets and distinct non-singleton sets, there will be two adjacent vertices in each $C_{m_i}$ are mono-indexed. Label the vertices of each copy, in such a way that the corresponding edges of neighbouring copies $C_{m_i}$ must not be mono-indexed. Then, there will be one mono-indexed edge between $C_{m_i}$ and $C_{m_{i+1}}$ for all $i<n$. Therefore, there are $n+1$ mono-indexed edges, one in each copy $C_{m_i}$ and $n$ mono-indexed edges, connecting $C_{m_i}$ and $C_{m_{i+1}}$. Therefore, $\varphi(C_m\times P_n) = 2n+1$    
\end{proof}

The following theorem discusses the sparing number of the cartesian product $C_m\times C_n$ of two cycles $C_m$ and $C_n$.

\begin{theorem}\label{T-NWFN}
Let $C_m$ and $C_n$ be two cycles. Then, the sparing number of the product $C_m\times C_n$ is 
\[ \varphi(C_m\times C_n) = 
  \begin{cases}
    $0$ & \quad \text{if both $m$ and $n$ are even}\\
    $2n$ & \quad \text{if $m$ is odd and $n$ is even}\\
    $2l$ & \quad \text{otherwise}~~ l=max\, (m,n).
  \end{cases} \]
\end{theorem}
\begin{proof}
Let $C_{m_i}$ be the $i$-th copy of $C_m$ in $C_m\times C_n$. Label the vertices of $C_{m_i}$, for odd values of $i$, starting from the initial vertex, by distinct singleton sets and distinct non-singleton sets (that are not used for labeling vertices in any other copy of $C_m$), alternately and label the vertices of $C_{m_i}$, for odd values of $i$, starting from the initial vertex, by distinct non-singleton sets and distinct singleton sets (that are not used for labeling vertices in any other copy of $C_m$) alternately in such a way that no two adjacent vertices are labeled by non-singleton sets.
Now we have the following cases.

{\em Case-1:} If both $C_m$ and $C_n$ are even, then by Remark \ref{R-CP2G}, the product $C_m\times C_n$ is the union of even cycles and hence is bipartite. Hence, by Theorem \ref{T-WUC}, the sparing number of $C_m\times C_n$ is $0$.

\noindent {\em Case-2:} If $m$ and $n$ are not simultaneously even.

Here we have the following subcases.

\noindent  {\em Subcase-2.1:} Without loss of generality, let $C_m$ be an odd cycle and $C_n$ be an even cycle. Then each copy of $C_m$ must have at least one mono-indexed edge. That is, in each copy of $C_m$, at least two adjacent vertices are mono-indexed. Therefore, there exist at least one mono-indexed edge between two neighbouring copies $C_{m_i}$ and $C_{m_{i+1}}$, for all $i<n$. Therefore, the total number of mono-indexed edges in $C_m\times C_n$ is $2n$.

\noindent  {\em Subcase-2.2:} Without loss of generality, let $m\le n$. Let both $C_m$ and  $C_n$ be two odd cycles. Then, in $C_m\times C_n$, in each copy of $C_m$, at least two adjacent vertices are mono-indexed. Therefore, there exist at least one mono-indexed edge between two neighbouring copies $C_{m_i}$ and $C_{m_{i+1}}$, for all $i<n$. Therefore, the total number of mono-indexed edges in $C_m\times C_n$ is $2n$.

Now, let $ m\ge n$. Then, since $C_m\times C_n$ and $C_n\times C_m$ are isomorphic graphs, $C_m\times C_n$ can be considered as the graph consisting of $m$ copies of $C_n$ with the corresponding edges of consecutive two copies are joined by edges. Hence, as explained in the above paragraph, the total number of mono-indexed edges in $C_m\times C_n$ is $2m$. That is, the sparing number of $C_m\times C_n$ is $max(m,n)$, if $m$ and $n$ are odd.
\end{proof}

An interesting question in this context is about the sparing number of the cartesian product of two graphs, at least one of which is a complete graph. The following theorem estimates the sparing number of the cartesian product of two complete graphs. 

\begin{theorem}\label{T-NME2K}
The sparing number of the product $K_m\times K_n$ of two complete graphs $K_m$ and $K_n$ is 
\begin{equation*}
\varphi(K_m\times K_n)=
\begin{cases}
n\binom{m-1}{2}+m\binom{n-1}{2} & ~~ \text{if}~~ n<m\\
m(m-1)(m-2) & ~~ \text{if}~~ n=m\\
\frac{1}{2} m(n-2)(m+n-2) & ~~ \text{if} ~~ n>m.
\end{cases}
\end{equation*}
\end{theorem}
\begin{proof}
Let $v_{ij}$ be the $i$-th vertex of the $j$-th copy of $K_m$ in $K_m\times K_n$. Then, $u_{ij}$ is adjacent to all other vertices in the same copy of $K_m$ and is adjacent to the corresponding vertices of all other copies of $K_m$ in $K_m\times K_n$. Therefore, the degree of $u_{ij}$ is $m+n-2$. That is, $K_m\times K_n$ is a $(m+n-2)$-regular graph. More over, the number of vertices in $K_m\times K_n$ is $mn$. Hence, the number of edges in $K_m\times K_n$ is $\frac{1}{2}mn(m+n-2)$. Here, we have the following cases. Also, each copy of $K_m$ has at most one vertex that is not mono-indexed.

\noindent {\em Case 1:} Let $n<m$. Then, each copy of $K_m$ has at most $(m-1)$ edges that are not mono-indexed. More over, $(n-1)$ edges that are not mono-indexed, are incident on one vertex of each copy of $K_m$. Therefore, the maximum number of edges that are not mono-indexed in $K_m\times K_n$ is $m(n-1)+n(m-1)$. Hence, the number of mono-indexed edges in $K_m\times K_n$ is
\begin{eqnarray*}
\varphi(K_m \times K_n) & = & \frac{1}{2}mn(m+n-2)-[m(n-1)+n(m-1)] \\
& = & \frac{1}{2}[m^2n+mn^2-6mn+2m+2n]\\
& = & \frac{1}{2} [m(n-1)(n-2)+n(m-1)(m-2)]\\
& = & n\binom{m-1}{2}+m\binom{n-1}{2}.
\end{eqnarray*}

\noindent {\em Case 2:} Let $n=m$. Then, by Case 1, $\varphi(K_m \times K_n)= 2m\binom{m-1}{2}=m(m-1)(m-2)$. 

\noindent {\em Case 3:} Let $n>m$. Then, $m$ copies of $K_m$ have one mono-indexed vertex each and the remaining $(n-m)$ copies must be $1$-uniform. Since the corresponding vertices of all copies of $K_m$ are adjacent to each other, no two corresponding vertices can have non-singleton set-labels. Therefore, the total number of edges that are not mono-indexed in $K_m\times K_n$ is $m(m+n-2)$. Therefore, the number of mono-indexed edges in $K_m\times K_n$ is
\begin{eqnarray*}
\varphi(K_m \times K_n) & = & \frac{1}{2}mn(m+n-2)-m(m+n-2) \\
& = & \frac{1}{2}[m^2n+mn^2-4mn-2m^2+4m]\\
& = & \frac{1}{2} m(n-2)(m+n-2).
\end{eqnarray*}
This completes the proof.
\end{proof}

We now proceed to determine the sparing number of the cartesian product of a complete graph and a path.

\begin{theorem}\label{T-KnxPm}
The sparing number of the cartesian product of a complete graph $K_n$ and a path $P_m$ is  $\frac{1}{2}(n-1)[(m+1)(n+1)-2]$.
\end{theorem}
\begin{proof}
The path $P_m$ has $m+1$ vertices, we have $m+1$ copies of $K_n$ in $K_n\times P_m$. By Theorem \ref{T-WKN}, one vertex of each copy of $K_n$ can have at most one vertex that is not mono-indexed. Also, note that the corresponding vertices of the $i$-th and $(i+1)$-th copies are adjacent in $K_n\times P_m$ and hence can not have non-singleton set-labels simultaneously. Let $u_{ij}$ be the $i$-th vertex of the $j$-th copy of $K_n$. Then, for odd  values of $j$, label the vertex $u_{1j}$ by distinct non-singleton sets and for even values of $j$, label the vertex $u_{2j}$ by distinct non-singleton sets. 

Now, by Theorem \ref{T-WKN}, each copy of $K_n$ has $\frac{1}{2}(n-1)(n-2)$ mono-indexed edges. Here, for $1\le j\le m$, the edges $u_{1,j}u_{1,j+1}$ and $u_{2,j}u_{2,j+1}$ have non-singleton set-labels. That is, there are $(n-2)$ mono-indexed edges connecting the $j$-th and $(j+1)$-th copy of $K_n$. Therefore, the total number of mono-indexed edges in $K_n\times P_m$ is $\frac{1}{2}(m+1)(n-1)(n-2)+m(n-2)= \frac{1}{2}[m(n+1)+(n-1)]=\frac{1}{2}(n-1)[(m+1)(n+1)-2]$
\end{proof}

In the following theorem, we estimate the sparing number of the cartesian product of a cycle and a complete graph.

\begin{theorem}
The sparing number of the cartesian product of a complete graph $K_n$ and a cycle $C_m$ is 
\begin{equation*}
\varphi(K_n\times C_m)=
\begin{cases}
\frac{1}{2}m(n+1)(n-2) &  \text{if $m$ is even}\\
\frac{1}{2}(n+1)[m(n-2)+2] &  \text{if $m$ is odd}.
\end{cases}
\end{equation*}
\end{theorem}
\begin{proof}
Here, we consider the following cases.

\noindent {\em Case-1:} Let $m$ be even. Then, as mentioned in the proof of Theorem \ref{T-KnxPm}, label the vertex $u_{1j}$ by non-singleton sets, for odd  values of $j$ and label the vertex $u_{2j}$ by non-singleton sets for even values of $j$. Therefore, as explained in Theorem \ref{T-KnxPm}, the total number of mono-indexed edges is  $m\frac{1}{2}(n-1)(n-2)+m(n-2)=\frac{1}{2}m(n+1)(n-2)$.

\noindent {\em Case-2:} Let $m$ be odd. Then, $m-1$ copies of  $K_n$ can be labeled as in Case-1 and $m$-th copy must be $1$-uniform. There is exactly one edge between the $m$-th copy and first copy of $K_n$ as well as the $m$-th copy and $(m-1)$-th copy of $K_n$, that is not mono-indexed. Therefore, the number of mono-indexed edges in $K_n\times P_n$ is $(m-1)\frac{1}{2}(n-1)(n-2)+(m-2)(n-2)+2(n-1)+\frac{1}{2}n(n-1)=\frac{1}{2}(n+1)[m(n-2)+2]$.
\end{proof}

In the following discussions, we intend to investigate about the sparing number of the cartesian product of two graphs, at least one of which is a complete bipartite graph. If both the graphs are bipartite, then their cartesian product will also be a bipartite graph and hence its sparing number is $0$. Hence, we need not study the cases when the second graph is a path or an even cycle. Therefore, we examine the sparing number of $K_{m_1, m_2} \times C_n$ where $n$ is an odd integer in the following theorem.

\begin{theorem}
For any odd integer $n$ and for the integers $m_1\le m_2$, the sparing number of $K_{m_1, m_2} \times C_n$ is $m_1(m_2+1)$.
\end{theorem}
\begin{proof}
Let $(X,Y)$ be the bipartition of $K_{m_1, m_2}$ with $|X|=m_1$ and $|Y|=m_2$. Let $X_i$ and $Y_i$ be the corresponding bipartitions of $K_{m_1, m_2}$ in $K_{m_1, m_2} \times C_n$. Now, label all the vertices of $X_i$ by distinct singleton sets and the vertices of $Y_i$ by distinct non-singleton sets for odd values of $i$ and label all the vertices of $X_i$ by distinct non-singleton sets and the vertices of $Y_i$ by distinct singleton sets for even values of $i$. Then, in the first $n-1$ copies all the corresponding vertices have different types (singleton and non-singleton sets) of set-labels and hence have no mono-indexed edges between them. But, the set-labels of the corresponding vertices of the $n$-th copy and the first copy can not be of different type unless one of them is $1$-uniform. Hence, assume that $m$-th copy of $K_{m_1, m_2}$ is $1$-uniform. Therefore, besides all the edges of  $n$-th copy of $K_{m_1, m_2}$, the edges between the partitions $X_1$ and $X_n$ are also mono-indexed.
the number of mono-indexed vertices in $K_{m_1, m_2} \times C_n$ is $m_1m_2+m_1=m_1(m_2+1)$.
\end{proof}

We, now proceed to determine the sparing number of the cartesian product of a complete graph $K_n$ and a complete bipartite graph $K_{m_1,m_2}$.

\begin{theorem}
The sparing number of $K_{m_1,m_2}\times K_n$ is $(n-1)m_1m_2+\frac{1}{2}n[nm_1+(n-2)m_2]$.
\end{theorem}
\begin{proof}
Let $G=K_{m_1,m_2}\times K_n$. Then, $G$ contains $n$ copies of $K_{m_1,m_2}$ with the corresponding vertices of all copies are adjacent to each other. Then, since no two adjacent vertices can have non-singleton set-labels, only one copy of $K_{m_1,m_2}$ can have a partition of vertices having non-singleton set-labels. That is, $(n-1)$ copies of $K_{m_1,m_2}$ are $1$-uniform in $G$. More over, no edge of the first copy of $K_{m_1,m_2}$ is $1$-uniform.

Let $(X,Y)$ be the bipartition of $K_{m_1,m_2}$ and let $(X_i,Y_i)$ be the corresponding bipartition of its $i$-th copy. Therefore, $|X_i|=|X|=m_1$ and $|Y_i|=|Y|=m_2$, where $1\le i\le n$. Then, the number of vertices in all $X_i$ is $m_1n$ and the number of vertices in all $Y_i$ is $m_2n$. For $1\le i\le n$, degree of a vertex in $X_i$ is $m_2+n$ and the sum of degrees of vertices of $X_i$ in each copy is $m_1(m_2+n)$. Therefore, the total degree of vertices in all $X_i$ in $G$ is $n.m_1(m_2+n)$. Similarly, the total degree of vertices in all $Y_i$ in $G$ is $n.m_2(m_1+n)$. Therefore, the total number of edges in $G$ is $\frac{1}{2}[n.m_1(m_2+n)+n.m_2(m_1+n)]$.

Let $m_1\le m_2$. Without loss of generality, let $Y_1$ be the set of vertices of $G$ having non-singleton set-labels. Then, the number of vertices that are not mono-indexed is the sum of degrees of vertices in $Y_1$. That is, number of vertices that are not mono-indexed in $G$ is $m_2(m_1+n)$. 

Therefore, the number of mono-indexed edges in $G$ is $\frac{1}{2}[n.m_1(m_2+n)+n.m_2(m_1+n)]-m_2(m_1+n) = (n-1)m_1m_2+\frac{1}{2}n[nm_1+(n-2)m_2]$.
\end{proof}

\section{Conclusion}

In this paper, we have discussed about the sparing number of cartesian products of certain graphs which admit weak IASIs. Some problems in this area are still open. We have not studied about the sparing number of the cartesian product of two arbitrary graphs $G_1$ and $G_2$, in our present discussions. Uncertainty in the adjacency pattern of different graphs makes this study complex. An investigation to determine the sparing number of the cartesian product of two arbitrary graphs in terms of their orders, sizes and the number of odd cycles in each of them, seems to be fruitful. The admissibility of weak IASIs by other graph products is also worth studying.

\end{document}